\documentclass[12pt]{amsart}

\usepackage[margin=1.25in]{geometry}

\usepackage{amsrefs}
\usepackage{amssymb}

\usepackage{hyperref}

\allowdisplaybreaks

\newtheorem{theorem}[equation]{Theorem}
\newtheorem{corollary}[equation]{Corollary}
\newtheorem{lemma}[equation]{Lemma}

\newcommand{\eref}[1]{(\ref{e.#1})}
\newcommand{\tref}[1]{Theorem \ref{t.#1}}
\newcommand{\cref}[1]{Corollary \ref{c.#1}}
\newcommand{\lref}[1]{Lemma \ref{l.#1}}

\numberwithin{equation}{section}
 
\newcommand{\Z}{\mathbb{Z}}
\newcommand{\R}{\mathbb{R}}
\newcommand{\C}{\mathbb{C}}
\newcommand{\T}{\mathbb{T}}
\renewcommand{\H}{\mathbb{H}}
\newcommand{\E}{\mathbb{E}}
\newcommand{\V}{\mathbb{V}}
\renewcommand{\P}{\mathbb{P}}
\renewcommand{\Im}{\operatorname{Im}}
\renewcommand{\Re}{\operatorname{Re}}
\newcommand{\id}{\operatorname{1}}
\newcommand{\ep}{\varepsilon}

\begin{document}

\title{The regular tree Anderson model at low disorder}

\author{Reuben Drogin}

\author{Charles K Smart}

\begin{abstract}
We prove delocalization for the Anderson model on an infinite regular tree (or Cayley graph or Bethe lattice) at low disorder.  This extends earlier results of Klein and Aizenman--Warzel by filling in the previously missing parts of the spectrum.  Our argument generalizes to any disorder with small fourth moment and sufficiently regular density.  We prove continuity of the Lyapunov exponent as the disorder vanishes.
\end{abstract}

\maketitle

\section{Introduction}

\subsection{Delocalization results}

The Anderson model on the regular tree is the only homogeneous random Schr\"odinger operator for which delocalization at low disorder is proved.  The tree structure provides an exact self-consistent equation for the law of the Green's function.  The exact equation makes the tree model easier to analyze than its analogue on the Euclidean lattice $\Z^d$ for $d \geq 2$.  Nonetheless, there are still unresolved questions and this paper proves new delocalization results.

Let $\T$ be an infinite regular tree of degree $K+1 \geq 3$.  Let $H : \ell^2(\T) \to \ell^2(\T)$ be the random Hamiltonian defined by
\begin{equation*}
(H f)(p) = \sum_{q: q \sim p} f(q) - V(p) f(p)
\end{equation*}
where $V : \T \to \R$ is an independent and identically distributed random potential and $q \sim p$ means $p$ and $q$ are neighboring vertices.  Assume there is a disorder strength $\beta > 0$ and a regularity bound $L \geq 1$ such that
\begin{equation}
\label{e.fourthmoment}
\E V(p)^4 \leq \beta^4,
\end{equation}
and
\begin{equation}
\label{e.regularity}
\frac{\P(|V(p) - x| < t)}{\P(|V(p) - x| < s)} \leq \frac{L t}{s} \quad \mbox{for } x \in \R \mbox{ and } 0 < t < s < \beta.
\end{equation}
These conditions are satisfied by many random laws of interest, like uniform and Gaussian, but exclude heavy-tailed or singular laws like Cauchy or Bernoulli.

By Pastur's Theorem, the spectrum of $H$ is almost surely the sum of the interval $[-2 \sqrt K, 2 \sqrt K]$ and the support of the law of $V(p)$.  When the disorder is small, the spectrum of $H$ can be partitioned into localized and delocalized regions. 

\begin{theorem}
\label{t.spectral}
For every $\ep > 0$, there is a $\delta > 0$ depending only on $(K,L,\ep)$ such that, if $\beta < \delta$, then, almost surely, the spectral measure of $H$ is absolutely continuous in $\{ E \in \R : |E| < K + 1 - \ep \}$ and pure point in $\{ E \in \R : |E| > K + 1 + \ep \}$.
\end{theorem}

When the random potential is bounded, the entire spectrum is delocalized and the result can be stated in purely dynamical terms.

\begin{corollary}
\label{c.dynamical}
There is a $\delta > 0$ depending only on $(K,L)$ such that, if $\beta < \delta$ and $\P(|V(p)| \leq \beta) = 1$, then, almost surely,
\begin{equation}
\label{e.dynamical}
\lim_{t \to \infty} \sum_{p \in S} |(e^{\pm i t H} f)(p)|^2 = 0
\end{equation}
holds for every $f \in \ell^2(\T)$ and finite $S \subseteq \T$.
\end{corollary}

The results in \tref{spectral} and \cref{dynamical} are new for low disorder and energies in the region $2 \sqrt K - \ep < |E| < K + 1 - \ep$.  Other regions of the energy-disorder phase space were covered by prior work.  For large disorder, Aizenman--Molchanov \cite{MR1244867} proved pure point spectrum for all energies by introducing the fractional moment method.  For small disorder, Aizenman \cite{MR1301371} proved pure point spectrum for energies $|E| > K + 1 + \ep$ by perturbing off of the free Green's function to obtain fractional moment bounds.    On the delocalization side, there are a series of results for small disorder.  Klein \cite{MR1302384} proved absolutely continuous spectrum for energies $|E| < 2 \sqrt K - \ep$ using supersymmetry to derive second moment bounds for the Green's function.  Froese--Hasler--Spitzer \cite{MR2547606} gave a new proof of Klein's result using dynamics on the upper half-plane.  Aizenman--Warzel \cite{MR3055759} introduced a new criterion for delocalization and used it to prove absolutely continuous spectrum for energies $\beta - o(\beta) < |E| - 2 \sqrt K < \beta$ when the potential is bounded.  They also proved dense presence of absolutely continuous spectrum for energies $|E| < K + 1 - \ep$ when the potential has full support. Shortly after, Bapst \cite{MR3390777} used this criterion to prove that for a fixed $E$, $K$ large and disorder strength proportional to $K\log K$, there is a crossover from at least some continuous spectrum to full localization near $E$.

There is a large mathematical literature on the Anderson model on tree graphs.   The monograph Aizenman--Warzel \cite{MR3364516} is a good starting point.  Of the more recent work, the most immediately related is the preprint of Aggarwal--Lopatto \cite{arXiv:2503.08949}, who proved that in a similar large disorder, large $K$ regime as Bapst, the energy-disorder phase space is separated into finitely many regions of localized and delocalized spectrum. There is also a proof of quantum unique ergodicity in Anantharaman--Sabri \cite{MR3707062} for the bulk of the spectrum.  Moving away from the homogeneous model, there are results on localization and delocalization for random trees and random $d$-regular graphs, where the randomness has been moved from the potential into the graph structure itself.  This includes Arras--Bordenave \cite{MR4645723}, Alt--Ducatez--Knowles \cites{MR4328063, MR4691859}, and Bauerschmidt--Huang--Yau \cites{MR4692880, MR3962004}. 

\subsection{The Aizenman--Warzel criterion}

The proof of \tref{spectral} relies on the Aizenman--Warzel criterion \cite{MR3055759} for delocalization.  This criterion is based on a certain Lyapunov exponent.

Given a vertex $p \in \T$ and energy $z \in \H$ in the upper half-plane $\H$, let $\pi_p : \ell^2(\T) \to \ell^2(\T \setminus \{ p \})$ denote the natural projection, let
\begin{equation*}
G^p_z = (\pi_p (H - z) \pi_p^*)^{-1}
\end{equation*}
denote the punctured Green's function, and let
\begin{equation}
\label{e.lyapunov}
\lambda_z = \E \log |G^p_z(q,q)|
\end{equation}
for any $p \sim q \in \T$ denote the Lyapunov exponent.  Since $H$ is self-adjoint, both the punctured Green's function and the Lyapunov exponent are well-defined and qualitatively bounded.  The resolvent identity implies the factorization
\begin{equation*}
G^{p_0}_z(p_1,p_n) = \prod_{k = 1}^n G^{p_{k-1}}_z(p_k,p_k)
\end{equation*}
for all simple paths $p_0, ..., p_n \in \T$.  Computing logarithmic moments gives
\begin{equation*}
\E \log |G^{p_0}_z(p_1,p_n)| = n \lambda_z
\end{equation*}
and shows the Lyapunov exponent measures the rate of decay of the Green's function off of the diagonal.

The delocalization criterion compares the rate of decay of the Green's function against the rate of growth of the metric ball in the tree.

\begin{theorem}[Aizenman--Warzel \cite{MR3055759}]
\label{t.criterion}
If $I \subseteq \R$ is a bounded open interval and
\begin{equation}
\label{e.criterion}
\liminf_{\eta \to 0} \inf_{E \in I} \lambda_{E + i \eta} > - \log K,
\end{equation}
then the spectral measure of $H$ in $I$ is almost surely absolutely continuous.
\end{theorem}

In light of this criterion and the prior work mentioned above, to prove \tref{spectral} it suffices to estimate the Lyapunov exponent.

\subsection{Green's function estimates}

The main contribution of this paper is the following Green's function estimate.

\begin{theorem}
\label{t.green}
There are constants $\alpha > 1 > \delta > 0$ depending only on $(K,L)$ such that, if $E \in \R$, $0 < \beta < \delta$, $0 < \eta < \beta^2$, and $p \sim q \in \T$, then
\begin{equation}
\label{e.green}
\P(|G^p_{E+i \eta}(q,q) - w_{E + i \eta}| \geq \alpha \beta^{1/20} |w_{E + i \eta}|) \leq \alpha \beta^{1/20} |w_{E + i \eta}|,
\end{equation}
where
\begin{equation}
\label{e.wz}
w_z = \frac{- z + \sqrt{ z + 2 \sqrt K } \sqrt{ z - 2 \sqrt K}}{2K}
\end{equation}
and $\sqrt{e^z} = e^{z/2}$ for $- \pi < \Im z \leq \pi$.
\end{theorem}

The estimate \eref{green} implies the punctured Green's function converges in probability to the free punctured Green's function as the disorder vanishes.  For fixed $E + i \eta$, this convergence is straightforward.  The uniform algebraic rate is new.

\subsection{The self-consistent equation}

The proof of \tref{green} proceeds by estimating the solutions of the following self-consistent equation for the punctured Green's function.  Suppose $p \in \T$ has neighbors $p_0, ..., p_K \in \T$, $h = V(p)$, $g_0 = G^{p_0}_{E + i \eta}(p,p)$, and $g_k = G^p_{E + i \eta}(p_k,p_k)$ for $k = 1, ..., K$.  The resolvent identity implies the self-consistent equation
\begin{equation}
\label{e.selfconsistent}
g_0 = \frac{-1}{g_1 + \cdots + g_K + h + E + i \eta}.
\end{equation}

Since the random variables $g_0, .., g_K$ are identical and $g_1, ..., g_K, h$ are independent, the self-consistent equation can be interpreted as a nonlinear convolution equation for the law of the $g_k$.  Let $M_1(\H)$ and $M_1(\R)$ denote the spaces of Borel probability measures on upper half-plane and real line, respectively.  If $\mu \in M_1(\H)$ is the law of the $g_k$ and $\nu \in M_1(\R)$ is the law of $h$, then the self-consistent equation \eref{selfconsistent} is equivalent to the convolution equation
\begin{equation}
\label{e.convolution}
\mu = (\mu^{*K} * \nu * \delta_{E + i \eta}) \circ \psi
\end{equation}
where $\mu^{*1} = \mu$, $\mu^{*(k+1)} = \mu * \mu^{*k}$, and $\psi(z) = -1/z$.

Another suggestive way to rewrite the self-consistent equation is
\begin{equation}
\label{e.phi}
g_0 = \phi_E \left( \frac{g_1 + \cdots + g_K + h + i \eta}{K} \right),
\end{equation}
where
\begin{equation*}
\phi_E(z) = \frac{-1}{K z + E}
\end{equation*}
is a M\"obius transformation of the upper half-plane.  Since $w_E = \phi_E(w_E)$, the estimate \eref{green} amounts to proving $g_0$ is concentrated near a particular fixed point of $\phi_E$.  This suggests the dynamics of the M\"obius transformation $\phi_E$ should play some role.  The proof below breaks into cases according to whether the dynamics are elliptic, parabolic, or hyperbolic.

\subsection{Additional hypotheses}

The hypotheses \eref{fourthmoment} and \eref{regularity} together imply that $V(p)$ has a small expected value and its law has a density with quadratic tails.  It is convenient to make versions of these into explicit hypotheses, as this simplifies some technical details below without changing any theorem statements above.  Henceforth assume (without loss of generality) the potential is mean zero,
\begin{equation}
\label{e.meanzero}
\E V(p) = 0
\end{equation}
and has sub-Cauchy tails,
\begin{equation}
\label{e.subcauchy}
\P(|V(p) - x| < t) \leq \frac{L \beta t}{\beta^2 + x^2} \quad \mbox{for } x \in \R \mbox{ and } 0 < t < \beta.
\end{equation}

The proof of \tref{green} below relies only on the hypotheses \eref{fourthmoment}, \eref{meanzero}, and \eref{subcauchy}.  In fact, the elliptic and parabolic cases use only \eref{fourthmoment} and \eref{meanzero} while the hyperbolic case uses only \eref{subcauchy}.  The regularity condition \eref{regularity} is only needed to apply \tref{criterion}.

\subsection{Acknowledgements}

The second author was partially supported by NSF grant DMS-2137909.

\section{Hyperbolic distance}

The upper half-plane has a natural hyperbolic metric structure for which M\"obius transformations are automorphisms.  This was used extensively by Froese--Hasler--Spitzer \cite{MR2547606}, who considered the hyperbolic distance
\begin{equation*}
d(z, w) = \frac{|z - w|^2}{\Im z \Im w} \quad \mbox{for } z, w \in \H.
\end{equation*}
While the hyperbolic distance is a monotone function of the natural hyperbolic metric, it is not itself a metric.  The triangle inequality fails for $d$ in general.  However, the loss of the triangle inequality is compensated by a gain of convexity and smoothness.

M\"obius transformations of the complex plane which map the upper half-plane into itself are contractions with respect to hyperbolic distance.

\begin{lemma}
If $\phi : \C \to \C$ is M\"obius and $\phi(\H) \subseteq \H$, then
\begin{equation}
\label{e.contraction}
d(\phi(z),\phi(w)) \leq d(z, w).
\end{equation}
\end{lemma}

\begin{proof}
Any such M\"obius transformation is a composition of maps of the form $z \mapsto -1/z$, $z \mapsto z + w$ for $w \in \H$, and $z \mapsto a z$ for $a > 0$.  By direct computation, each of these three types is a contraction with respect to hyperbolic distance.
\end{proof}

Hyperbolic distance is convex in each variable separately and quasiconvex in both variables jointly.

\begin{lemma}
For all $z_1, z_2, w_1, w_2 \in \H$,
\begin{equation}
\label{e.convex}
d\left( \frac{z_1+z_2}{2}, w_1 \right) \leq \frac{d(z_1, w_1) + d(z_2, w_1)}{2}
\end{equation}
and
\begin{equation}
\label{e.quasiconvex}
d\left( \frac{z_1+z_2}{2}, \frac{w_1 + w_2}{2} \right) \leq \max \{ d(z_1, w_1), d(z_2, w_2) \}.
\end{equation}
\end{lemma}

\begin{proof}
For \eref{convex} compute
\begin{align*}
& d(z_1, w_1) + d(z_2, w_1) - 2 d(z_1 + z_2, 2 w_1) \\
& = \frac{| (z_1 - w_1) \Im z_2 - (z_2 - w_1) \Im z_1 |^2}{\Im z_1 \Im z_2 \Im (z_1 + z_2) \Im w_1} \geq 0
\end{align*}
and conclude by scaling invariance.  For \eref{quasiconvex}, compute
\begin{align*}
& \sqrt{d\left( z_1 + z_2, w_1 + w_2 \right)} \\
& = \frac{|z_1 + z_2 - w_1 - w_2|}{\sqrt{\Im z_1 + \Im z_2} \sqrt{\Im w_1 + \Im w_2}} \\
& \leq \frac{|z_1 - w_1| + |z_2 - w_2|}{\sqrt{\Im z_1 + \Im z_2} \sqrt{\Im w_1 + \Im w_2}} \\
& = \sqrt{d(z_1,w_1)} \sqrt t \sqrt s + \sqrt{d(z_2,w_2)} \sqrt{1-t} \sqrt{1-s}
\end{align*}
where
\begin{equation*}
t = \frac{\Im z_1}{\Im z_1 + \Im z_2}
\end{equation*}
and
\begin{equation*}
s = \frac{\Im w_1}{\Im w_1 + \Im w_2}
\end{equation*}
and conclude using $\sqrt t \sqrt s + \sqrt{1-t} \sqrt{1 - s} \leq 1$.
\end{proof}

The ratios of the imaginary parts of two points in the upper half-plane are bounded by their hyperbolic distance.

\begin{lemma}
If $z, w \in \H$, then
\begin{equation}
\label{e.iratio}
\frac{\Im z}{\Im w} + \frac{\Im w}{\Im z} \leq 2 + d(z,w).
\end{equation}
\end{lemma}

\begin{proof}
Compute
\begin{equation*}
d(z,w) \geq \frac{(\Im z - \Im w)^2}{\Im z \Im w} = \frac{\Im z}{\Im w} + \frac{\Im w}{\Im z} - 2
\end{equation*}
and rearrange.
\end{proof}

The second moment of hyperbolic distance perturbed by a random real number is controlled by the fourth moment of the perturbation.

\begin{lemma}
There is a universal constant $C > 1$ such that, if $z, w \in \H$, $s \in [0,1]$, $X \in \R$ random, $\E X= 0$, and $\E X^4 \leq s^4 (\Im w)^4$, then
\begin{equation}
\label{e.perturb}
\E d(z + X, w)^2 \leq \left( 1 + C s^2 \right) d(z, w)^2 + C s^2.
\end{equation}
\end{lemma}

\begin{proof}
Throughout the proof use $C > 1$ to denote a universal constant that may differ in each instance.
For $u \in \C$, compute
\begin{equation*}
|X - u|^4 = |u|^4 - 4 |u|^2 (\Re u) X + 2 |u|^2 X^2 + 4 (\Re u)^2 X^2 - 4 (\Re u) X^3 + X^4.
\end{equation*}
Take the expectation and use $\E X = 0$ and $|\Re u| \leq |u|$ to obtain
\begin{equation*}
\E |X - u|^4 \leq |u|^4 + C |u|^2 \E X^2 + C |u| \E |X|^3 + \E X^4.
\end{equation*}
Apply H\"older's inequality to obtain
\begin{equation*}
\E |X - u|^4 \leq |u|^4 + C |u|^2 \E X^2 + C \E X^4.
\end{equation*}
Use the definition of $d(z,w)$ to obtain
\begin{equation*}
\E d(z + X,w)^2
= \frac{\E |X + z - w|^4}{(\Im w \Im z)^2}
\leq d(z,w)^2 + C d(z,w) \frac{\E X^2}{\Im w \Im z} + C \frac{\E X^4}{(\Im w \Im z)^2}.
\end{equation*}
Use the hypothesis $\E X^4 \leq s^4 (\Im w)^4$ and the ratio bound \eref{iratio} to obtain
\begin{equation*}
\E d(z + X, w)^2 \leq d(z,w)^2 + C s^2 d(z,w) (2 + d(z,w)) + C s^4 (2 + d(z,w))^2.
\end{equation*}
Use $s^4 \leq s^2$ and expand to obtain
\begin{equation*}
\E d(z + X, w)^2 \leq d(z,w)^2 + C s^2 d(z,w)^2  + C s^2 d(z,w) + C s^2.
\end{equation*}
Conclude \eref{perturb} using Cauchy--Schwarz.
\end{proof}

A crude version of the triangle inequality holds for hyperbolic distance.

\begin{lemma}
If $z, w_1, w_2\in \H$, then
\begin{equation}
\label{e.almosttriangle}
\sqrt{d(z, w_1)} \leq (1 + \sqrt{d(w_2,w_1)}) \sqrt{d(z,w_2)} + \sqrt 2 \sqrt{d(w_2,w_1)}.
\end{equation}
\end{lemma}

\begin{proof}
Compute
\begin{align*}
& \sqrt{ d(z, w_1) } \\
& \leq \frac{|z - w_2| + |w_2 - w_1|}{\sqrt{\Im z \Im w_1}} \\
&  = \sqrt{ d(z, w_2)} + \frac{|w_2 - w_1|}{\sqrt{\Im z \Im w_1}} \\
&  \leq \sqrt{ d(z, w_2)} + \sqrt{ d(w_2,w_1) } \sqrt{ 2 + d(z, w_1)) } \\
&  \leq (1 + \sqrt{ d(w_2,w_1) })\sqrt{ d(z, w_2)} + \sqrt 2 \sqrt{ d(w_2, w_1) },
\end{align*}
using the triangle inequality and \eref{iratio}. 
\end{proof}

The convexity of hyperbolic distance can be quantified.  This essentially comes from the level sets of $d(\cdot, w)$ being nested disks.

\begin{lemma}
If $z_1, ..., z_K, w \in \H$, then
\begin{equation}
\label{e.rconvex}
d\left( \frac{z_1 + \cdots + z_K}{K}, w \right) \leq  \left( 1 - \frac{\sum_{j < k} |z_j - z_k|^2}{K \sum_k |z_k - w|^2} \right) \max_k d(z_k, w)
\end{equation}
\end{lemma}

\begin{proof}
Compute
\begin{align*}
& d\left(\frac{z_1 + \cdots + z_K}{K}, w \right) \\
& = \frac{|\sum_k (z_k - w)|^2}{K \Im w \sum_k \Im z_k} \\
& = \frac{|\sum_k (z_k - w)|^2}{K \sum_k |z_k - w|^2 d(z_k,w)^{-1}} \\
& \leq \frac{|\sum_k (z_k - w)|^2}{K \sum_k |z_k - w|^2} \max_k d(z_k, w) \\
& = \left( 1 - \frac{\sum_{j < k} |z_j - z_k|^2}{K \sum_k |z_k - w|^2} \right) \max_k d(z_k, w)
\end{align*}
using the triangle inequality and algebra.
\end{proof}

\section{Fixed point iteration}

The following lemma shows the convolution equation \eref{convolution} has a unique solution.  Moreover, the associated fixed-point iteration has a globally stable fixed point.

\begin{lemma}
\label{l.iteration}
Suppose $E \in \R$, $\eta > 0$, $\nu \in M_1(\R)$, $\psi(z) = -1/z$, and $T : M_1(\H) \to M_1(\H)$ is defined by
\begin{equation*}
T \mu = (\mu^{*K} * \nu * \delta_{E + i \eta}) \circ \psi.
\end{equation*}
For any $\mu \in M_1(\H)$, the fixed-point iteration $k \mapsto T^k \mu$ converges weakly to the unique fixed point of $T$.
\end{lemma}

\begin{proof}
Let $W(\mu_1, \mu_2)$ denote the $\infty$-Wasserstein distance between $\mu_1, \mu_2 \in M_1(\H)$ with respect to hyperbolic distance $d$.   Even though $d$ is not a metric, the definition of $W$ still makes sense because $d$ is a monotone function of a metric.  Since the topology on $M_1(\H)$ generated by $W$ is stronger than the weak topology, it suffices to prove convergence of fixed-point iteration with respect to $W$.

Using the contraction \eref{contraction} of M\"obius transformations with respect to hyperbolic distance, compute
\begin{equation*}
W(T^2\mu_1, T^2\mu_2) \leq W((T\mu_1)^{* K} * \nu * \delta_{E + i\eta}, (T\mu_2)^{* K} * \nu * \delta_{E + i\eta}).
\end{equation*}
Using the invariance of hyperbolic distance with respect to translation by real numbers, obtain
\begin{equation*}
W(T^2\mu_1, T^2\mu_2) \leq W((T\mu_1)^{* K} * \delta_{i \eta}, (T\mu_2)^{* K} * \delta_{i \eta}).
\end{equation*}
Using $\eta > 0$ and the definition of $T$, observe
\begin{equation*}
\operatorname{supp} T \mu_1, \operatorname{supp} T \mu_2  \subseteq  \{ z \in \H : |z| \leq \eta^{-1} \}.
\end{equation*}
Since $d(z + i \eta, w + i \eta) \leq 4 R^2 \eta^{-2}$ whenever $z, w \in \H$ and $|z|, |w| \leq R$, obtain
\begin{equation*}
W((T \mu_1)^{* K} * \delta_{i \eta}, (T \mu_2)^{* K} * \delta_{i \eta}) \leq 4 K^2 \eta^{-4}.
\end{equation*}
Since $d(z + i \eta, w + i \eta) \leq (\frac{R}{R + \eta})^2 d(z,w)$ whenever $z, w \in \H$ and $|z|, |w| \leq R$, obtain
\begin{equation*}
W((T\mu_1)^{* K} * \delta_{i \eta}, (T\mu_2)^{* K} * \delta_{i \eta}) \leq \left( \frac{K}{K+\eta^2} \right)^2 W((T \mu_1)^{* K}, (T \mu_2)^{* K}).
\end{equation*}
Using the quasiconvexity \eref{quasiconvex} of hyperbolic distance, obtain
\begin{equation*}
W((T \mu_1)^{* K}, (T \mu_2)^{* K}) \leq W( T \mu_1, T \mu_2 ).
\end{equation*}

Chaining the above inequalities together gives
\begin{equation*}
W(T^2 \mu_1, T^2 \mu_2) \leq \min \left\{ \frac{4 K^2}{\eta^4}, \left( \frac{K}{K+\eta^2} \right)^2 W(T \mu_1, T \mu_2) \right\}
\end{equation*}
for any $\mu_1, \mu_2 \in M_1(\H)$.  This implies convergence of the fixed-point iteration.  The distance $W(T^k \mu_1, T^k \mu_2)$ becomes bounded after the first step and decays geometrically afterwards.
\end{proof}

To prove the hyperbolic case of \tref{green}, it is useful to project the convolution equation \eref{convolution} onto the real line.  A projection from $M_1(\H)$ onto $M_1(\R)$ can be defined using Cauchy measure.  Let $\sigma_z \in M_1(\R)$ denote the Cauchy measure with barycenter $z \in \H$ defined by
\begin{equation*}
d \sigma_{a + i b}(x) = \frac{1}{\pi} \frac{b}{b^2 + (x - a)^2} \,dx.
\end{equation*}
Let $\sigma_\mu \in M_1(\R)$ denote the Cauchy projection of $\mu \in M_1(\H)$ defined by
\begin{equation}
\label{e.projection}
\sigma_\mu = \int_\H \sigma_z \,d \mu(z).
\end{equation}

The Cauchy measure $\sigma_z$ is the hitting measure of a Brownian motion in the complex plane started at $z \in \H$ and stopped upon hitting the real line. Thus, the Cauchy projection $\sigma_\mu$ can be interpreted as the hitting measure of a Brownian motion whose starting point is randomly sampled from the distribution $\mu$. By conformal invariance of Brownian motion, the Cauchy projection commutes with convolution and automorphisms of $\H$.  The following lemma and its elementary proof encode these facts.

\begin{lemma}\label{l.projection}
If $\mu_1, \mu_2 \in M_1(\H)$ and $\phi$ is an automorphism of $\H$, then
\begin{equation*}
\sigma_{\mu_1 * \mu_2} = \sigma_{\mu_1} * \sigma_{\mu_2},
\end{equation*}
and
\begin{equation*}
\sigma_{\mu \circ \phi} = \sigma_\mu \circ \phi.
\end{equation*}
\end{lemma}

\begin{proof}
Since
\begin{equation*}
d \sigma_z(x) = \frac{1}{2 \pi} \left( \frac{1}{x - z} - \frac{1}{x - \bar z} \right) dx,
\end{equation*}
the residue theorem implies
\begin{equation*}
\int f(x) \,d \sigma_z(x) = f(z)
\end{equation*}
whenever $f : \overline \H \to \C$ is bounded and holomorphic.  By the density of such $f$ this identity characterizes $\sigma_z \in M_1(\R)$.  Compute
\begin{equation*}
\int f(x) \,d (\sigma_z * \sigma_w)(x) = \int \int f(x + y) \,d \sigma_z(x) \,d \sigma_w(y) = f(z + w)
\end{equation*}
and deduce
\begin{equation*}
\sigma_z * \sigma_w = \sigma_{z+w}.
\end{equation*}
Compute
\begin{equation*}
\int f(x) \,d (\sigma_z \circ \phi)(x) = \int f(\phi^{-1}(x)) \,d \sigma_z(x) = f(\phi^{-1}(z))
\end{equation*}
and deduce
\begin{equation*}
\sigma_z \circ \phi = \sigma_{\phi^{-1}(z)}.
\end{equation*}
Conclude the lemma by inserting the above identities into the definition \eref{projection} of Cauchy projection.
\end{proof}

\section{The elliptic and parabolic cases}

This section covers the elliptic and parabolic cases of \tref{green}.  It is convenient to impose the following standing hypotheses throughout the section.  Assume $0 < \beta < 1$, $0 < \eta < \beta^2$, and $|E| \leq K + 1$.  Assume $g_0, ..., g_K \in \H$ and $h \in \R$ are the random variables solving the self-consistent equation \eref{selfconsistent} and its dynamical reformulation \eref{phi} defined in the introduction.  Finally, $C > 1 > c > 0$ denote constants that depend only on $(K,L)$ but may differ in each instance.

The goal is to prove $g_0$ is typically close to $w_E$ in a suitably strong sense.  The starting point is the same as Froese--Hasler--Spitzer \cite{MR2547606} who used the second moment
\begin{equation*}
\E d(g_0, w)^2
\end{equation*}
for some $w \in \H$ as an approximate Lyapunov functional for the dynamics of \eref{phi}.

The first lemma shows this second moment is qualitatively finite.

\begin{lemma}
For all $w \in \H$,
\begin{equation*}
\E d(g_0, w)^2 < \infty.
\end{equation*}
\end{lemma}

\begin{proof}
The self-consistent equation \eref{selfconsistent} and $\eta > 0$ gives
\begin{equation*}
|g_0| \leq \eta^{-1}.
\end{equation*}
Use $\E h^4 \leq \beta^4$ to compute
\begin{align*}
& \E (\Im g_0)^{-2} \\
& = \E |g_0|^{-4} (\Im g_0^{-1})^{-2} \\
& \leq \eta^{-2} \E |g_1 + \cdots + g_K + h + E + i \eta|^4 \\
& \leq C \eta^{-2} (\eta^{-1} + \beta + |E| + \eta)^4.
\end{align*}
Compute
\begin{equation*}
\E d(g_0, w)^2 \leq (\eta^{-1} + |w|)^4 (\Im w)^{-2} \E (\Im g_0)^{-2}.
\end{equation*}
Chaining these together gives the lemma.
\end{proof}

The next lemma shows the hyperbolic distance $d(g_0, w)$ is relatively concentrated when $w$ is sufficiently invariant under the dynamics.  The qualitative convexity of hyperbolic distance together with the averaging in the self-consistent equation play an essential role here.

\begin{lemma}
If $w \in \H$, $\Im w > \eta$, and
\begin{equation*}
\sqrt{d(\phi_E(w),w)} + \sqrt{d(w - i \eta/K, w)} + (\Im w)^{-2} \beta^2 \leq s \leq 1,
\end{equation*}
then the variance $\V d(g_0, w)$ of $d(g_0,w)$ satisfies
\begin{equation}
\label{e.dconcentrate}
\V d(g_0, w) \leq C s (1 + \E d(g_0, w)^2).
\end{equation}
\end{lemma}

\begin{proof}
The following elementary observation is used throughout the proof: Any estimate of the form $A \leq (1 + C s) B + C s$ implies an estimate of the form $A^2 \leq (1 + C s) B^2 + C s$.  This uses Cauchy--Schwarz, $0 \leq s \leq 1$, and the convention the constants $C > 1 > c > 0$ are allowed to differ in each instance.

Compute
\begin{align*}
& \E d(g_0, w)^2 \\
& = \E d\left( \phi_E \left( \frac{g_1 + \cdots + g_K + h + i \eta}{K} \right), w \right)^2 \\
& \leq (1+Cs) \E d\left( \phi_E \left( \frac{g_1 + \cdots + g_K + h + i \eta}{K} \right), \phi_E(w) \right)^2 + Cs \\
& \leq (1+Cs) \E d\left( \frac{g_1 + \cdots + g_K + h}{K}, w - \frac{i \eta}{K} \right)^2 + Cs \\
& \leq (1+Cs) \E d\left( \frac{g_1 + \cdots + g_K + h}{K}, w \right)^2 + Cs \\
& \leq (1+Cs) \E d\left( \frac{g_1 + \cdots + g_K }{K}, w \right)^2 + Cs \\
& \leq (1+Cs) \E \left( \frac{d(g_1,w) + \cdots + d(g_K,w) }{K} \right)^2 + Cs \\
& = (1+Cs) \left( \E d(g_0, w)^2 - \frac{K-1}{K} \V d(g_0, w) \right) + Cs,
\end{align*}
using the self-consistent equation \eref{selfconsistent} in the first step, the triangle inequality \eref{almosttriangle} and $\sqrt{d(\phi_E(w),w)} \leq s$ in the second step, the contraction property \eref{contraction} in the third step, the triangle inequality \eref{almosttriangle} and $\sqrt{d(w - i \eta/K, w)} \leq s$ in the fourth step, the independence of $h$ from $g_1, ..., g_K$, $\E h^4 \leq \beta^4 \leq s^2 (\Im w)^4$, and the perturbation estimate \eref{perturb} in the fifth step, the convexity of hyperbolic distance \eref{convex} in the sixth step, and the identical distributions of $g_0, ..., g_K$ and independence of $g_1, ..., g_K$ in the seventh step.

The variance estimate \eref{dconcentrate} follows by rearranging.
\end{proof}

The next lemma uses the the self-consistent equation and the quantitative convexity to deduce relative concentration of the difference $g_0 - w$ from relative concentration of the hyperbolic distance $d(g_0, w)$.

\begin{lemma}
If $w \in \H$, $\Im w > \eta$,
\begin{equation*}
d(\phi_E(w),w) + (\Im w)^{-2} \beta^2 \leq r \leq 1,
\end{equation*}
\begin{equation*}
\V d(g_0, w) \leq r \E d(g_0, w)^2,
\end{equation*}
and
\begin{equation*}
\bar d = \E d(g_0, w) > 0,
\end{equation*}
then
\begin{equation}
\label{e.zconcentrate}
\P( |g_2 - g_1| \geq t r^{1/4} (1 + \bar d^{-1/2}) |g_1 - w| ) \leq C t^{-4} \quad \mbox{for } C \leq t \leq c r^{-1/4}.
\end{equation}
\end{lemma}

\begin{proof}
Assume $r > 0$ is small since otherwise the conclusion is vacuous.

Use $\E d(g_0, w)^2 = \bar d^2 + \V d(g_0, w)$ and $\V d(g_0, w) \leq r \E d(g_0, w)^2$ to conclude
\begin{equation*}
\E d(g_0, w)^2 \leq C \bar d^2
\end{equation*}
and
\begin{equation*}
\V d(g_0, w) \leq C r \bar d^2.
\end{equation*}

For $1 \leq t \leq r^{-1/2}$, let $\mathcal E_t$ denote the event
\begin{equation*}
|h| \leq t \beta \quad \mbox{and} \quad |d(g_k, w) - \bar d| \leq t r^{1/2} \bar d \quad \mbox{for } k = 0, ..., K.
\end{equation*}

Since $\E h^2 \leq \beta^2$ and $\E |d(g_k,w) - \bar d|^2 \leq C r \bar d^2$, Markov's inequality implies
\begin{equation*}
\P(\mathcal E_t^c) \leq C t^{-2}.
\end{equation*}
On the event $\mathcal E_t$, compute
\begin{align*}
& (1 - t r^{1/2}) \bar d \\
& \leq d(g_0, w) \\
& = d\left( \phi_E \left( \frac{g_1 + \cdots + g_K + h + i \eta}{K} \right), w \right) \\
& \leq (1 + C r^{1/2}) d\left( \phi_E \left( \frac{g_1 + \cdots + g_K + h + i \eta}{K} \right), \phi_E(w) \right) + C r^{1/2} \\
& \leq (1 + C r^{1/2} ) d\left( \frac{g_1 + \cdots + g_K}{K}, w - \frac{h + i \eta}{K} \right) + C r^{1/2} \\
& \leq (1 + C t r^{1/2}) d\left( \frac{g_1 + \cdots + g_K}{K}, w \right) + C t r^{1/2}  \\
& \leq (1 + C t r^{1/2}) \left( 1 - \frac{\sum_{1 \leq j < k \leq K} |g_j - g_k|^2}{K \sum_{1 \leq k \leq K} |g_k - w|^2} \right) \max_{1 \leq k \leq K} d(g_k, w) + C t r^{1/2} \\
& \leq (1 + C t r^{1/2}) \left( 1 - \frac{\sum_{1 \leq j < k \leq K} |g_j - g_k|^2}{K \sum_{1 \leq k \leq K} |g_k - w|^2} \right) \bar d + C t r^{1/2},
\end{align*}
using the definition of $\mathcal E_t$ in the first step, the self-consistent equation \eref{selfconsistent} in the second step, the almost triangle inequality \eref{almosttriangle} and $d(\phi_E(w),w) \leq r$ in the third step, the contraction property \eref{contraction} in the fourth step, the almost triangle inequality \eref{almosttriangle} and $\sqrt{d(w - (h + i \eta)/K, w)} \leq C (\Im w)^{-1} t \beta \leq C t r^{1/2}$ in the fifth step, the strict convexity \eref{rconvex} in the sixth step, and the definition of $\mathcal E_t$ in the seventh step.

Dividing the above inequality by $\bar d$ gives
\begin{equation*}
\P \left( \frac{\sum_{1 \leq j < k \leq K} |g_j - g_k|^2}{\sum_{1 \leq k \leq K} |g_k - w|^2} \geq t r^{1/2} (1 + \bar d^{-1}) \right) \leq \P(\mathcal E_t^c) \leq C t^{-2} \quad \mbox{for } C \leq t \leq c r^{-1/2}.
\end{equation*}

Conclude \eref{zconcentrate} by the triangle inequality.
\end{proof}

The next lemma covers the elliptic case of \tref{green}.  The proof combines the relative concentration from the above two lemmas with the elliptic dynamics.  The argument also recovers a version of the second moment bound on the Green's function obtained by Klein \cite{MR1302384}.

\begin{lemma}
\label{l.stronglyelliptic}
If $\beta \leq c$ and $|E| \leq 2 \sqrt K - C \sqrt \beta$, then
\begin{equation}
\label{e.secondmoment}
\E d(g_0, w_E)^2 \leq C \beta^{1/6}
\end{equation}
and
\begin{equation}
\label{e.stronglyelliptic}
\P( |g_0 - w_E| \geq C \beta^{1/9} |w_E| ) \leq C \beta^{1/9} |w_E|.
\end{equation}
\end{lemma}

\begin{proof}
Assume $\alpha \beta \leq 1$ and $|E| \leq 2 \sqrt K - \sqrt{\alpha \beta}$ for some $\alpha \geq 1$ to be determined.

Taylor expanding the definition \eref{wz} of $w_E$ gives
\begin{equation*}
c (\alpha \beta)^{1/4} \leq \Im w_E \leq C
\end{equation*}
and
\begin{equation*}
c \leq |w_E| \leq C.
\end{equation*}
The second moment bound \eref{secondmoment} and the concentration bound \eref{stronglyelliptic} follow from the slightly stronger second moment bound
\begin{equation}
\label{e.secondmomentprime}
\E d(g_0, w_E)^2 \leq C (\Im w_E)^{-2} \beta^{2/3}.
\end{equation}
Assuming \eref{secondmomentprime}, Markov's inequality gives
\begin{align*}
& \P(|g_0 - w_E| \geq \beta^{1/9}) \\
& \leq \beta^{-2/9}\E |g_0-w_E|^2 \\
& = (\Im w_E)\beta^{-2/9}  \E(\Im g_0)d(g_0,w_E) \\
& \leq (\Im w_E)\beta^{-2/9} (\E d(g_0,w_E)^2)^{1/2} (\E|g_0|^2)^{1/2}\\
&\leq C \beta^{1/9},
\end{align*}
using Cauchy-Schwarz in the third step and the estimate $1+d(z,w_E)\geq c|z|$ with \eref{secondmomentprime} in the fourth step.
Thus the concentration bound \eref{stronglyelliptic} follows from the stronger second moment bound \eref{secondmomentprime}.

To prove \eref{secondmomentprime}, it is safe to assume
\begin{equation*}
\E d(g_0, w_E)^2 \geq (\Im w_E)^{-2} \beta^{2/3}.
\end{equation*}

Since $\phi_E(w_E) = w_E$ and $\eta < \beta^2$, obtain
\begin{equation*}
\sqrt{d(\phi_E(w_E), w_E)} + \sqrt{d(w_E - i \eta/K, w_E)} + (\Im w_E)^{-2} \beta^2 \leq C (\Im w_E)^{-2} \beta^2.
\end{equation*}
The variance bound \eref{dconcentrate} then gives
\begin{equation*}
\V d(g_0, w_E) \leq C (\Im w_E)^{-2} \beta^2 (1 + \E d(g_0, w_E)^2).
\end{equation*}
By the assumed lower bound on $\E d(g_0, w_E)^2$, this implies
\begin{equation*}
\V d(g_0, w_E) \leq C \beta^{4/3} \E d(g_0, w_E)^2.
\end{equation*}

Since
\begin{equation*}
d(\phi_E(w_E), w_E) + (\Im w_E)^{-2} \beta^2 \leq C \beta^{4/3},
\end{equation*}
the concentration bound \eref{zconcentrate} gives
\begin{equation*}
\P( |g_2 - g_1| \geq t (\Im w_E)^{1/2} \beta^{1/6} |g_1 - w_E| ) \leq t^{-4} \quad \mbox{for } C \leq t \leq c \beta^{-1/3}.
\end{equation*}

Now we exploit the elliptic dynamics of $\phi_E$. Introduce a copy $\tilde g$ of $g_0$ that is independent from the random variables $g_0, ..., g_K, h$.  By Markov's inequality and the above concentration bound, the event
\begin{equation*}
\mathcal E =
\begin{cases}
|h| \leq C \beta \\
|g_k - \tilde g| \leq C (\Im w_E)^{1/2} \beta^{1/6} |\tilde g - w_E| & \mbox{for } k = 0, ..., K \\
d(g_0, w_E)^2 \geq c \E d(g_0, w_E)^2 \\
\end{cases}
\end{equation*}
holds with positive probability. On the event $\mathcal E$, compute
\begin{align*}
& \frac{|g_0 - w_E|^2}{|g_0|} \\
& \leq \frac{| g_0 - w_E| |g_0 - \bar w_E|}{|g_0|} \\
& = |g_0 - \phi_E^{-1}(g_0)| \\
& = \left|g_0 - \frac{g_1 + \cdots + g_K + h + i \eta}{K}\right| \\
& \leq C (\Im w_E)^{1/2} \beta^{1/6} |\tilde g - w_E| + C \beta \\
& \leq C (\Im w_E)^{1/2} \beta^{1/6} |g_0 - w_E| + C \beta.
\end{align*}
Using $\beta \leq c$ small, the above inequality implies
\begin{equation*}
|g_0 - w_E| \leq C (\Im w_E)^{1/2} \beta^{1/6}.
\end{equation*}
Imposing the bound $\alpha \geq C$ gives
\begin{equation*}
\Im g_0 \geq \Im w_E - |g_0 - w_E| \geq c \Im w_E
\end{equation*}
and
\begin{equation*}
d(g_0, w_E) \leq C (\Im w_E)^{-1} \beta^{1/3}.
\end{equation*}
Therefore, on an event of positive probability,
\begin{equation*}
c \E d(g_0, w_E)^2 \leq d(g_0, w_E)^2 \leq C (\Im w_E)^{-2} \beta^{2/3}.
\end{equation*}
Since the left and right hand sides in the above inequality are deterministic, the second moment bound \eref{secondmoment} holds.
\end{proof}

The next lemma covers the parabolic case of \tref{green}.  The proof is similar to that of the previous lemma, but with two major differences.  First, since the invariant point $w_E$ of $\phi_E$ may lie on the real line, a perturbation is used as a base point for the approximate Lyapunov function.  Second, it is not possible to bound second moment $\E d(g_0, w)^2$ uniformly for any choice of $w \in \H$ as $\eta \to 0$.

\begin{lemma}
\label{l.weaklyparabolic}
If $\alpha \geq 1$, $\alpha \beta \leq c$, and $||E| - 2 \sqrt K| \leq \sqrt{\alpha \beta}$, then
\begin{equation}
\label{e.weaklyparabolic}
\P(|g_0 - w_E| \geq C (\alpha \beta)^{1/20}) \leq C (\alpha \beta)^{1/20}.
\end{equation}
\end{lemma}

\begin{proof}
Let
\begin{equation*}
w = w_E + i (\alpha \beta)^{1/4}.
\end{equation*}
Taylor expanding the definitions of $w_E$ and $\phi_E$, obtain
\begin{equation*}
|\phi_E(w) - w| \leq C (\alpha \beta)^{1/2}
\end{equation*}
and
\begin{equation*}
(\alpha \beta)^{1/4} \leq \Im w \leq C (\alpha \beta)^{1/4}.
\end{equation*}

Since $|g_0 - w| \geq (\alpha \beta)^{1/6}$ implies $d(g_0, w) \geq c (\alpha \beta)^{-1/12}$, Markov's inequality gives
\begin{equation*}
\P(|g_0 - w_E| \geq (\alpha \beta)^{1/6}) \leq C (\alpha \beta)^{1/6} \E d(g_0, w)^2.
\end{equation*}
In particular, it is safe to assume
\begin{equation*}
\E d(g_0, w)^2 \geq 1
\end{equation*}
as otherwise the conclusion \eref{weaklyparabolic} holds.

Since
\begin{equation*}
\sqrt{d(w, \phi(w))} + \sqrt{d(w - i \eta/K, w)} + (\Im w)^{-2} \beta^2 \leq C (\alpha \beta)^{1/4},
\end{equation*}
the variance bound \eref{dconcentrate} gives
\begin{equation*}
\V d(g_0, w) \leq C (\alpha \beta)^{1/4} (1 + \E d(g_0, w)^2).
\end{equation*}
Since $\E d(g_0, w)^2 \geq 1$, obtain
\begin{equation*}
\V d(g_0, w) \leq C (\alpha \beta)^{1/4} \E d(g_0, w)^2.
\end{equation*}

Since
\begin{equation*}
d(w, \phi(w)) + (\Im w)^{-2} \beta^2 \leq C (\alpha \beta)^{1/2} \leq C (\alpha \beta)^{1/4} \E d(g_0, w)^2,
\end{equation*}
the concentration bound \eref{zconcentrate} gives
\begin{equation*}
\P( |g_2 - g_1| \geq t (\alpha \beta)^{1/16} |g_1 - w| ) \leq C t^{-4} \quad \mbox{for } C \leq t \leq c (\alpha \beta)^{-1/16}.
\end{equation*}

Let $\tilde g$ be a copy of $g_0$ independent from $g_0, ..., g_K, h$.  For $t > 0$, let $\mathcal E_t$ denote the event
\begin{equation*}
\begin{cases}
|h| \leq t \beta \\
|g_k - \tilde g| \leq t (\alpha \beta)^{1/16} |\tilde g - w| & \mbox{for } k = 0, ..., K.
\end{cases}
\end{equation*}
By Markov's inequality and the concentration bound,
\begin{equation*}
\P( \mathcal E_t^c ) \leq C t^{-4} \quad \mbox{for } C \leq t \leq c (\alpha \beta)^{-1/16}.
\end{equation*}

On the event $\mathcal E_t$, compute
\begin{align*}
& \frac{|g_0 - w_E|^2 - C (\alpha \beta)^{1/2}}{|g_0|}  \\
& \leq \frac{|g_0 - w_E| |g_0 - K^{-1} w_E^{-1}|}{|g_0|} \\
& = |g_0 - \phi_E^{-1}(g_0)| \\
& = \left| g_0 - \frac{g_1 + \cdots + g_K + h + i \eta}{K} \right| \\
& \leq C t (\alpha \beta)^{1/16} |\tilde g - w| + C t \beta \\
& \leq C t (\alpha \beta)^{1/16} |g_0 - w_E| + C t \beta.
\end{align*}
Assuming $\alpha \beta \leq c$ and $t (\alpha \beta)^{1/16} \leq c$ are sufficiently small, the above inequality implies
\begin{equation*}
|g_0 - w_E| \leq C t (\alpha \beta)^{1/16}.
\end{equation*}
Therefore, it follows that
\begin{equation*}
\P( |g_0 - w_E| \geq t (\alpha \beta)^{1/16} ) \leq C t^{-4} \quad \mbox{for } C \leq t \leq c (\alpha \beta)^{-1/16}.
\end{equation*}
Conclude \eref{weaklyparabolic} by setting $t = (\alpha \beta)^{-1/80} \in [C, c (\alpha \beta)^{-1/16}]$.
\end{proof}

\section{The hyperbolic case}

This section covers the hyperbolic case of \tref{green}.  It is convenient to impose the following standing hypotheses throughout the section.  Assume $0 < \beta < 1$, $0 < \eta < \beta^2$, and $|E| \geq 2 \sqrt K$.  Assume $\nu \in M_1(\R)$ is the law of $V(p)$ and $\mu \in M_1(\H)$ is the law of the punctured Green's functions $G^p_{E + i \eta}(q,q)$ for $p \sim q \in \T$.  Finally, $C > 1 > c > 0$ denote constants that depend only on $(K,L)$ but may differ in each instance.

The goal is to show most of the mass of $\mu$ lies near the fixed point $w_E$ of the M\"obius transformation $\phi_E$.  This will be accomplished through fixed-point iteration of the convolution equation \eref{convolution} and an appropriate choice of topology. Since the fixed-point iteration converges, it suffices to start from a good initial guess and prove good bounds are preserved during iteration.

The following lemma computes the Cauchy projection \eref{projection} of the fixed-point iteration onto the real line.  Let $D(\R)$ denote the space of probability densities on the real line.

\begin{lemma}
\label{l.densityiteration}
If $\mu_0 \in M_1(\H)$, $f_0 \in D(\R)$, $d \sigma_{\mu_0}(x) = f_0(x) \,dx$,
\begin{equation*}
\mu_{n+1} = (\mu_n^{*K} * \nu * \delta_{E + i \eta}) \circ \psi,
\end{equation*}
and
\begin{equation*}
f_{n+1} = ((f_n^{*K} * \nu * \delta_E * \sigma_{i \eta}) \circ \psi) \psi',
\end{equation*}
then $\mu_n \in M_1(\H)$, $f_n \in D(\R)$, and $d \sigma_{\mu_n}(x) = f_n(x) \,dx$.
\end{lemma}

\begin{proof}
Proceed by induction on $n$.  The base case is part of the hypothesis, so assume the conclusion holds for $n$.  Observe $f * \rho \in D(\R)$ whenever $f \in D(\R)$ and $\rho \in M_1(\R)$.  It follows from \lref{projection} that $f_n^{*K} * \nu * \delta_E * \sigma_{i \eta} \in D(\R)$ is the density of the Cauchy projection of $\ \mu_n^{*K} * \nu * \delta_{E + i \eta} \in M_1(\H)$.  Therefore, it suffices to prove $(f \circ \psi) \psi' \in D(\R)$ is the density of $\rho \circ \psi$ whenever $f \in D(\R)$ is the density of $\rho \in M_1(\R)$.  For a test function $g \in C_c(\R)$ compute
\begin{equation*}
\int g \,d(\rho \circ \psi) = \int (g \circ \psi^{-1}) \,d\rho = \int g(\psi^{-1}(x)) f(x) \,dx = \int g(y) f(\psi(y)) \psi'(y) \,dy,
\end{equation*}
using $\psi'(y) = 1/y^2 > 0$ and substitution.
\end{proof}

Given $f \in D(\R)$ and $s, r > 0$, consider sub-Cauchy tail bounds of the form
\begin{equation*}
f(x) \leq s (|x| - r)_+^{-2} = \begin{cases}
s (|x| - r)^{-2} & \mbox{if } |x| > r \\
\infty & \mbox{otherwise.}
\end{cases}
\end{equation*}
Roughly speaking, this says the measure given by the density $f$ is close to the Dirac measure $\delta_x$.  To see how bounds like this propagate under the fixed-point iteration, it suffices to estimate their behavior under convolution and pushforward.

The following lemma estimates the sub-Cauchy tails of a convolution of densities.

\begin{lemma}
If $s_1, s_2, r_1, r_2 > 0$, $f_1, f_2 \in D(\R)$, and $f_k(x) \leq s_k (|x| - r_k)_+^{-2},$ then
\begin{equation}
\label{e.convolutiondensity}
(f_1 * f_2)(x) \leq (s_1 + s_2 + 8 s_1 s_2 t^{-1}) (|x| - r_1 - r_2 - t)_+^{-2}
\end{equation}
holds for all $t > 0$.
\end{lemma}

\begin{proof}
Suppose $t > 0$ and $x > r_1 + r_2 + t$.  Compute
\begin{align*}
& (f_1 * f_2)(x) \\
& = \int_{-\infty}^{r_1 + t} f_1(y) f_2(y - x) \,dy \\
& \qquad + \int_{x - r_2 - t}^\infty f_1(y) f_2(y - x) \,dy \\
& \qquad \qquad + \int_{r_1+t}^{x - r_2 - t} f_1(y) f_2(y - x) \,dy \\
& \leq \sup_{y \leq r_1 + t} f_2(y - x) \\
& \qquad + \sup_{y \geq x - r_2 - t} f_1(y)  \\
& \qquad \qquad + \int_{r_1+t}^{x - r_2 - t} s_1 s_2 (y - r_1)^{-2} (x - r_2 - y)^{-2} \,dy \\
& \leq s_2 (x - r_1 - r_2 - t)^{-1} \\
& \qquad + s_1 (x - r_1 - r_2 - t)^{-1}  \\
& \qquad \qquad + 8 s_1 s_2 t^{-1} (x - r_1 - r_2)^{-2}.
\end{align*}
The case $-x > r_1 + r_2 + t$ is symmetric.
\end{proof}

The following lemma captures the contraction effect of the hyperbolicity of $\phi_E$.  The definition \eref{wz} of $w_E$ and the standing hypothesis $|E| \geq 2 \sqrt K$ imply $w_E \in \R$ and $|w_E| \leq 1/\sqrt K$.

\begin{lemma}
If $f_1 \in D(\R)$, $r, s > 0$, $|w_E| r < 1$, and $f_1(K w_E + x) \leq s (|x| - r)_+^{-2}$, then $f_2(x) = x^{-2} f_1(-x^{-1} - E)$ satisfies
\begin{equation}
\label{e.hyperboliccontraction}
f_2(w_E + x) \leq \tau s (|x| - \tau r)_+^{-2}
\end{equation}
where
\begin{equation*}
\tau = \frac{w_E^2}{(1 - |w_E| r)^2}.
\end{equation*}
\end{lemma}

\begin{proof}
Compute
\begin{align*}
& f_2(w_E + x) \\
& = \frac{1}{(w_E + x)^2} f_1\left( - \frac{1}{w_E + x} - E \right) \\
& \leq \frac{s}{(w_E + x)^2} \left( \left| \frac{-1}{w_E + x} - E - K w_E \right| - r \right)_+^{-2} \\
& = \frac{s}{(w_E + x)^2} \left( \left| \frac{-1}{w_E + x} + \frac{1}{w_E} \right| - r \right)_+^{-2} \\
& = w_E^2 s (|x| - |w_E| |w_E + x| r)_+^{-2} \\
& \leq w_E^2 s ( (1 - |w_E| r) |x| - w_E^2 r)_+^{-2} \\
& = \frac{w_E^2}{(1 - |w_E| r)^2} s \left( |x| - \frac{w_E^2}{1 - |w_E|r} r\right)_+^{-2} \\
& \leq \tau s (  |x| - \tau r)_+^{-2},
\end{align*}
using the definition of $f_2$ in the first step, the hypothesis on $f_1$ in the second step, the quadratic equation $K w_E^2 + E w_E + 1 = 0$ in the third step, algebra in the fourth step, the triangle inequality in the fifth step, algebra in the sixth step, and the definition of $\tau$ in the seventh step.
\end{proof}

The following lemma covers the hyperbolic case of \tref{green}.

\begin{lemma}
\label{l.stronglyhyperbolic}
If $\beta \leq c$ and $|E| \geq 2 \sqrt K + C \sqrt \beta$, then
\begin{equation}
\label{e.stronglyhyperbolic}
\mu( |z - w_E| \geq C \beta^{1/4} |w_E|) \leq C \beta^{1/2} |w_E|.
\end{equation}
\end{lemma}

\begin{proof}
Assume $\alpha \beta \leq 1$ and $|E| \geq 2 \sqrt K + \sqrt{\alpha \beta}$ for some large $\alpha \geq 1$ to be determined.

Step 1.
Consider the fixed-point iteration defined in \lref{densityiteration} with
\begin{equation*}
\mu_0 = \delta_{w_E + i \beta w_E^2}
\end{equation*}
and
\begin{equation*}
f_0(x) \,dx = d \sigma_{\mu_0}(x).
\end{equation*}
The first goal is to prove
\begin{equation}
\label{e.ih}
f_n(w_E + x) \leq \beta^{3/4} w_E^2 (|x| -  \beta^{1/4} |w_E|)_+^{-2}
\end{equation}
by induction on $n$.

Step 2.
For the base case, use the density formula for Cauchy measure to compute
\begin{equation*}
f_0(w_E + x) = \frac{1}{\pi} \frac{\beta w_E^2}{\beta^2 w_E^4 + x^2} \leq \beta^{3/4} w_E^2 x^{-2} \leq \beta^{3/4} w_E^2 (|x| - \beta^{1/4} |w_E|)^{-2}
\end{equation*}
when $|x| > \beta^{1/4} |w_E|$ and $\beta^{1/4} \leq 1/\pi$.

Step 3.
For the induction step, assume \eref{ih} holds for $f_n$.

The sub-Cauchy hypothesis \eref{subcauchy} on the random potential implies the measure $\nu$ has density $h \in D(\R)$ that satisfies
\begin{equation*}
h(x) \leq C \beta ( |x| - \beta )_+^{-2}.
\end{equation*}
Since $\eta < \beta^2$, the Cauchy measure $\sigma_{i \eta}$ has density $g \in D(\R)$ that satisfies
\begin{equation*}
g(x) \leq C \beta^2 ( |x| - \beta^2 )_+^{-2}.
\end{equation*}

Using the densities $h$ and $g$ and the formula $\psi(z) = -1/z$, the definition of $f_{n+1}$ in \lref{densityiteration} can be re-written as
\begin{equation*}
f_{n+1}(x) = x^{-2} (f_n^{*K} * h * g)(-x^{-1} - E).
\end{equation*}

Using the induction hypothesis \eref{ih}, the tail bounds on $h$ and $g$, and the convolution density inequality \eref{convolutiondensity} with $t = \beta^{1/2} w_E^2$ obtain
\begin{equation*}
(f_n^{*K} * h * g)(K w_E + x) \leq s (|x| - r)_+^{-2}
\end{equation*}
where
\begin{equation*}
s = K \beta^{3/4} w_E^2 + C \beta w_E^2 + C \beta^{5/4}
\end{equation*}
and
\begin{equation*}
r = K \beta^{1/4} |w_E| + C \beta^{1/2} w_E^2 + C \beta.
\end{equation*}

Applying the pushforward contraction inequality \eref{hyperboliccontraction}, obtain
\begin{equation*}
f_{n+1}(w_E + x) \leq \tau s (|x| - \tau r)_+^{-2}
\end{equation*}
where
\begin{equation*}
\tau = \frac{w_E^2}{(1 - |w_E| r)^2}.
\end{equation*}

To close the induction, it suffices to prove
\begin{equation}
\label{e.ihclose}
\tau s \leq \beta^{3/4} w_E^2 \quad \mbox{and} \quad \tau r \leq \beta^{1/4} |w_E|.
\end{equation}

Taylor expanding the definition \eref{wz} of $w_E$ and using $|E| \geq 2 \sqrt K + \sqrt{\alpha \beta}$ and $\alpha \beta \leq 1$ gives
\begin{equation*}
|w_E| \leq K^{-1/2} - c (\alpha \beta)^{1/4}.
\end{equation*}
Break into cases according to the size of $w_E$:

First, if $|w_E| \geq \beta^{1/8}$, then
\begin{equation*}
s \leq (K + C \beta^{1/4}) \beta^{3/4} w_E^2,
\end{equation*}
\begin{equation*}
r \leq (K + C \beta^{1/4}) \beta^{1/4} |w_E|,
\end{equation*}
and
\begin{equation*}
\tau \leq w_E^2 (1 + C \beta^{1/4}) \leq K^{-1} + C \beta^{1/4} - c (\alpha \beta)^{1/4}.
\end{equation*}
Thus, if $\alpha \geq C$, then the good $c (\alpha \beta)^{1/4}$ term beats the bad $C \beta^{1/4}$ terms and \eref{ihclose} holds.

Second, if $|w_E| \leq \beta^{1/8}$, then
\begin{equation*}
s \leq C \beta,
\end{equation*}
\begin{equation*}
r \leq C \beta^{3/8},
\end{equation*}
and
\begin{equation*}
\tau \leq C w_E^2.
\end{equation*}
If $\beta \leq \alpha^{-1} \leq c$, then the extra $\beta^{1/4}$ beats the constants and \eref{ihclose} holds.

Step 4.  To conclude the lemma, observe the Cauchy measure satisfies
\begin{equation*}
\sigma_z([\Im z, \infty)) = \sigma_i([1,\infty)) > 0.
\end{equation*}
Use this to compute
\begin{align*}
& \mu_n(|z - w_E| \geq r) \\
& = \int \id_{|z - w_E| \geq r} \,d\mu_n(z) \\
& \leq (\sigma_i([1,\infty)))^{-1} \int \id_{|x - w_E| \geq r} \,d \sigma_{\mu_n}(x) \\
& = (\sigma_i([1,\infty)))^{-1} \int \id_{|x - w_E| \geq r} f_n(x) \,dx.
\end{align*}
In particular, for $r > \beta^{1/4} |w_E|$, the estimate \eref{ih} gives
\begin{equation*}
\mu_n(|z - w_E| \geq r) \leq C \int_{|x| > r} \beta^{3/4} w_E^2 (x - \beta^{1/4} |w_E|)_+^{-2} \,dx \leq C \beta^{3/4} w_E^2 (r - \beta^{1/4} |w_E|)^{-1}.
\end{equation*}
Setting $r = 2 \beta^{1/4} |w_E|$ yields
\begin{equation*}
\mu_n(|z - w_E| \geq 2 \beta^{1/4} |w_E|) \leq C \beta^{1/2} |w_E|.
\end{equation*}
Since \lref{iteration} implies $\mu_n \to \mu$ weakly, the conclusion \eref{stronglyhyperbolic} follows.
\end{proof}

\section{Delocalization}

This section proves the theorems from the introduction.  As in the previous two sections, it is convenient to let $C > 1 > c > 0$ denote constants that depend only on $K$ and $L$ but may differ in each instance.

\begin{proof}[Proof of \tref{green}]
Apply \lref{stronglyelliptic}, \lref{weaklyparabolic}, and \lref{stronglyhyperbolic}.  Choosing $\alpha \geq C$ large enough, and then $\beta \leq c$ small enough, the elliptic regime $|E| \leq 2 \sqrt K - C \beta^{1/2}$, parabolic regime $||E| - 2 \sqrt K| \leq \sqrt{\alpha \beta}$, and hyperbolic regime $|E| \geq 2 \sqrt K + C \beta^{1/2}$ cover all of the energies.
\end{proof}

An estimate of the Lyapunov exponent \eref{lyapunov} is required to invoke the Aizenman--Warzel criterion for delocalization.  This is obtained in two steps.  First, the sub-Cauchy density bound \eref{subcauchy} is used to obtain uniform weak integrability bounds on the punctured Green's function.

\begin{lemma}
\label{l.weak}
For all $p \sim q \in \T$, $z \in \H$, and $t > 0$,
\begin{equation*}
\P( |G^p_z(q,q)| \geq t^{-1} ) \leq C \beta^{-1} t
\end{equation*}
and
\begin{equation*}
\P( |G^p_z(q,q)| \leq (t^{-1} + |z|)^{-1} ) \leq C \beta^{-1} t.
\end{equation*}
\end{lemma}

\begin{proof}
Consider the random variables $g_0, ..., g_K \in \H$ and $h \in \R$ defined along with the self-consistent equation \eref{selfconsistent} in the introduction.  The sub-Cauchy density bound \eref{subcauchy} implies
\begin{equation*}
\P(|h| \leq \beta t) \leq C t
\end{equation*}
and
\begin{equation*}
\P(|h| \geq \beta t^{-1}) \leq C t,
\end{equation*}
Using the independence of $h$ from $g_1, ..., g_K$, obtain
\begin{equation*}
\P(|g_0| \geq t^{-1}) \leq \P( |g_1 + \cdots + g_K + h + z| \leq t) \leq C \beta^{-1} t.
\end{equation*}
Using the identical distributions of the $g_k$, obtain
\begin{align*}
& \P(|g_0| \leq (t^{-1} + |z|)^{-1}) \\
& = \P( |g_1 + \cdots + g_K + h + z| \geq t^{-1} + |z| ) \\
& \leq K \P( |g_0| \geq (K+1)^{-1} t^{-1}) + \P ( |h| \geq (K+1)^{-1} t^{-1}) \\
& \leq C \beta^{-1} t.
\end{align*}
Finally, recall that $g_0$ has the same law as $G^p_z(q,q)$ for $p \sim q \in \T$.
\end{proof}

The Lyapunov exponent can be controlled by interpolating the main Green's function estimate with the uniform weak integrability.

\begin{lemma}
\label{l.lyapunov}
If $E \in \R$, $0 < \eta \leq \beta^2$, and $\beta \leq c$, then the Lyapunov exponent \eref{lyapunov} satisfies
\begin{equation*}
|\lambda_{E + i \eta} - \log |w_E|| \leq C \beta^{1/21}.
\end{equation*}
\end{lemma}

\begin{proof}
Let $g = G^p_z(q,q)$.  By the Green's function estimate \eref{green},
\begin{equation*}
\P(|g - w_E| \geq C \beta^{1/20}|w_E|) \leq C \beta^{1/20} |w_E|.
\end{equation*}
Taylor expanding the definition \eref{wz} of $w_E$, obtain
\begin{equation*}
\frac{c}{1+|E|} \leq |w_E| \leq \frac{C}{1 + |E|}.
\end{equation*}
The weak bounds in \lref{weak} give
\begin{equation*}
\P (|g| \leq t |w_E|) \leq C \beta^{-1} t
\end{equation*}
and
\begin{equation*}
\P (|g| \geq t^{-1} |w_E|) \leq C \beta^{-1} |w_E|^{-1} t
\end{equation*}
for $t > 0$.  Interpolation gives
\begin{equation*}
|\E \log |g| - \log |w_E|| \leq C \beta^{1/20} |\log \beta|.
\end{equation*}
Conclude using $\beta \leq c$ small to absorb the logarithm.
\end{proof}

The delocalization results from the introduction can now be proved.

\begin{proof}[Proof of \tref{spectral}]
Since  Aizenman \cite{MR1301371} proved the spectral measure of $H$ is almost surely pure point in $\{ E \in \R : |E| > K + 1 + \ep \}$ when $\beta \leq c_\ep$, it suffices to prove the absolutely continuous part of the theorem.

From the definition \eref{wz} of $w_E$, obtain $1/\sqrt K \geq |w_E| > 1/K$ for $|E| < K+1$.  Moreover, since $|w_E| = 1/K$ and $\frac{d}{dE} |w_E| \neq 0$ for $E = \pm (K + 1)$, obtain $\log |w_E| \geq c \ep - \log K$ for $|E| \leq K + 1 - \ep$.  Therefore, when $|E| \leq K + 1 - \ep$, $\eta \leq \beta^2$, and $\beta \leq c \ep^{21}$,  \lref{lyapunov} implies $\lambda_{E + i \eta} \geq \log |w_E| - c \beta^{1/21} \geq c \ep - \log K$.  In particular, \eref{criterion} holds and the spectral measure of $H$ is almost surely absolutely continuous in $\{ E \in \R : |E| < K + 1 - \ep \}$.
\end{proof}

\begin{proof}[Proof of \cref{dynamical}]
When $\P(|V(p)| \leq \beta) = 1$, the spectrum of $H$ is almost surely contained in $[-2 \sqrt K - \beta, 2 \sqrt K + \beta]$.  Since this is strictly contained in $[-K-1,K+1]$ for small $\beta > 0$, \tref{spectral} implies the entire spectral measure is almost surely absolutely continuous when the disorder is small.  The decay bound \eref{dynamical} now follows by the RAGE (Ruelle–Amrein–Georgescu–Enss) Theorem.
\end{proof}


\begin{bibdiv}
\begin{biblist}

\bib{arXiv:2503.08949}{article}{
   author={Aggarwal, Amol},
   author={Lopatto, Patrick},
   title={Mobility Edge for the Anderson Model on the Bethe Lattice},
   journal={arXiv preprint},
   volume={arXiv:2503.08949},
   date={2025},
   note={\url{https://arxiv.org/abs/2503.08949}},
}

\bib{MR1301371}{article}{
   author={Aizenman, Michael},
   title={Localization at weak disorder: some elementary bounds},
   note={Special issue dedicated to Elliott H. Lieb},
   journal={Rev. Math. Phys.},
   volume={6},
   date={1994},
   number={5A},
   pages={1163--1182},
   issn={0129-055X},
   review={\MR{1301371}},
   doi={10.1142/S0129055X94000419},
}

\bib{MR1244867}{article}{
   author={Aizenman, Michael},
   author={Molchanov, Stanislav},
   title={Localization at large disorder and at extreme energies: an
   elementary derivation},
   journal={Comm. Math. Phys.},
   volume={157},
   date={1993},
   number={2},
   pages={245--278},
   issn={0010-3616},
   review={\MR{1244867}},
}

\bib{MR3055759}{article}{
   author={Aizenman, Michael},
   author={Warzel, Simone},
   title={Resonant delocalization for random Schr\"odinger operators on tree
   graphs},
   journal={J. Eur. Math. Soc. (JEMS)},
   volume={15},
   date={2013},
   number={4},
   pages={1167--1222},
   issn={1435-9855},
   review={\MR{3055759}},
   doi={10.4171/JEMS/389},
}

\bib{MR3364516}{book}{
   author={Aizenman, Michael},
   author={Warzel, Simone},
   title={Random operators},
   series={Graduate Studies in Mathematics},
   volume={168},
   note={Disorder effects on quantum spectra and dynamics},
   publisher={American Mathematical Society, Providence, RI},
   date={2015},
   pages={xiv+326},
   isbn={978-1-4704-1913-4},
   review={\MR{3364516}},
   doi={10.1090/gsm/168},
}

\bib{MR4328063}{article}{
   author={Alt, Johannes},
   author={Ducatez, Raphael},
   author={Knowles, Antti},
   title={Delocalization transition for critical Erd\"os-R\'enyi graphs},
   journal={Comm. Math. Phys.},
   volume={388},
   date={2021},
   number={1},
   pages={507--579},
   issn={0010-3616},
   review={\MR{4328063}},
   doi={10.1007/s00220-021-04167-y},
}

\bib{MR4691859}{article}{
   author={Alt, Johannes},
   author={Ducatez, Raphael},
   author={Knowles, Antti},
   title={Localized phase for the Erd\"os-R\'enyi graph},
   journal={Comm. Math. Phys.},
   volume={405},
   date={2024},
   number={1},
   pages={Paper No. 9, 74},
   issn={0010-3616},
   review={\MR{4691859}},
   doi={10.1007/s00220-023-04918-z},
}

\bib{MR3707062}{article}{
   author={Anantharaman, Nalini},
   author={Sabri, Mostafa},
   title={Quantum ergodicity for the Anderson model on regular graphs},
   journal={J. Math. Phys.},
   volume={58},
   date={2017},
   number={9},
   pages={091901, 10},
   issn={0022-2488},
   review={\MR{3707062}},
   doi={10.1063/1.5000962},
}

\bib{MR4645723}{article}{
   author={Arras, Adam},
   author={Bordenave, Charles},
   title={Existence of absolutely continuous spectrum for Galton-Watson
   random trees},
   journal={Comm. Math. Phys.},
   volume={403},
   date={2023},
   number={1},
   pages={495--527},
   issn={0010-3616},
   review={\MR{4645723}},
   doi={10.1007/s00220-023-04798-3},
}

\bib{MR2547606}{article}{
   author={Froese, Richard},
   author={Hasler, David},
   author={Spitzer, Wolfgang},
   title={Absolutely continuous spectrum for a random potential on a tree
   with strong transverse correlations and large weighted loops},
   journal={Rev. Math. Phys.},
   volume={21},
   date={2009},
   number={6},
   pages={709--733},
   issn={0129-055X},
   review={\MR{2547606}},
   doi={10.1142/S0129055X09003724},
}

\bib{MR1302384}{article}{
   author={Klein, Abel},
   title={Absolutely continuous spectrum in the Anderson model on the Bethe
   lattice},
   journal={Math. Res. Lett.},
   volume={1},
   date={1994},
   number={4},
   pages={399--407},
   issn={1073-2780},
   review={\MR{1302384}},
   doi={10.4310/MRL.1994.v1.n4.a1},
}

\bib{MR3962004}{article}{
   author={Bauerschmidt, Roland},
   author={Huang, Jiaoyang},
   author={Yau, Horng-Tzer},
   title={Local Kesten-McKay law for random regular graphs},
   journal={Comm. Math. Phys.},
   volume={369},
   date={2019},
   number={2},
   pages={523--636},
   issn={0010-3616},
   review={\MR{3962004}},
   doi={10.1007/s00220-019-03345-3},
}

\bib{MR4692880}{article}{
   author={Huang, Jiaoyang},
   author={Yau, Horng-Tzer},
   title={Spectrum of random $d$-regular graphs up to the edge},
   journal={Comm. Pure Appl. Math.},
   volume={77},
   date={2024},
   number={3},
   pages={1635--1723},
   issn={0010-3640},
   review={\MR{4692880}},
   doi={10.1002/cpa.22176},
}

\bib{MR3390777}{article}{
   author={Bapst, Victor},
   title={The large connectivity limit of the Anderson model on tree graphs},
   journal={J. Math. Phys.},
   volume={55},
   date={2014},
   number={9},
   pages={092101, 20},
   issn={0022-2488},
   review={\MR{3390777}},
   doi={10.1063/1.4894055},
}

\end{biblist}
\end{bibdiv}
\end{document}